\documentclass{siamart220329}

\usepackage{amssymb,amsmath,amsfonts,amscd,verbatim}
\usepackage{color,graphicx}
\usepackage{algorithm,algorithmic}
\usepackage{upgreek,comment,url}
\usepackage{mystylefile}
\usepackage{float}
\title{A randomized preconditioned\\ Cholesky-QR algorithm\thanks{The work of both authors 
was supported in part by NSF grant DMS-1760374.
The second author was also supported in part by NSF grant CCF-2209510, and DOE grant DE-SC0022085.}}
\author{James E. Garrison\thanks{Department of Mathematics, 
North Carolina State University, Raleigh, NC 27695-8205, USA,
\email{jegarri3@ncsu.edu}}
\and Ilse C.F. Ipsen\thanks{Department of Mathematics, 
North Carolina State University, Raleigh, NC 27695-8205, USA,
\email{ipsen@ncsu.edu}}}
\headers{Randomized Preconditioned Cholesky-QR}{J. E. Garrison and I. C. F. Ipsen}

\begin{document}
\maketitle

\begin{abstract}
We a present and analyze \textit{rpCholesky-QR},
a randomized preconditioned Cholesky-QR algorithm for computing the
thin QR factorization of real $m\times n$ matrices
with rank~$n$. \textit{rpCholesky-QR}
has a low orthogonalization error, a residual on the order of machine precision,
and does not break down for highly singular matrices.
We derive rigorous and interpretable two-norm perturbation 
bounds for \textit{rpCholesky-QR} that require a minimum of assumptions. 
Numerical experiments corroborate the accuracy of \textit{rpCholesky-QR} for preconditioners sampled from as few as $3n$ rows, and
illustrate that the two-norm deviation from orthonormality increases with only the condition 
number
of the preconditioned matrix, rather than its square ---even if the original matrix is numerically singular.
\end{abstract}



\section{Introduction}
Given a tall and skinny matrix $\ma\in\rmn$ with $\rank(\ma)=n\ll m$, we consider the computation
of a thin QR factorization $\ma=\mq\mr$ with a Cholesky-QR algorithm. 
Cholesky-QR algorithms compute
an explicit matrix $\mq\in\rmn$ with orthonormal columns ---as opposed to a full-fledged orthogonal matrix. Compared 
to other orthogonalization methods, Cholesky-QR algorithms can exhibit superior
performance on cache-based and parallel architectures \cite{StathWu2002},
because they are  high in BLAS-3 operations and 
perform well when communication (data movement, synchronization) dominates arithmetic \cite{Yama2015}. However, Cholesky-QR algorithms can be numerically unstable
or break down 
for less than perfectly conditioned matrices.

Applications of Cholesky QR algorithms include implementations of intermediate 
orthogonalization steps
in Krylov space methods \cite{Balabanov2022,BG2022,BG2023}, and high-performance
implementations of LOBPCG \cite{DSYG2018,KMS2023}.

\subsection{Contributions} 
Our proposed algorithm \textit{rpCholesky-QR} has the following advantages:
\begin{enumerate}
\item \textit{rpCholesky-QR} is a simple two-stage algorithm, with about the
same operation count as \textit{Cholesky-QR2} in \cite{YNYF2015}.
\item The number of rows sampled for the preconditioner can be as low as $3n$.
\item For well to moderately conditioned matrices, \textit{rpCholesky-QR} has the same high accuracy as \textit{Cholesky-QR2} in terms of deviation from orthonormality and residual.
\item \textit{rpCholesky-QR} works for highly ill-conditioned and numerically singular matrices, where
other Cholesky-QR algorithms break down. 
\item The two-norm residual of \textit{rpCholesky-QR} is always on the order of machine precision.
\item The two-norm 
deviation from orthonormality of \textit{rpCholesky-QR} tends to increase 
with only the condition number of the preconditioned
matrix rather than its square ---even if the original matrix $\ma$ is numerically singular.
\item The perturbation analysis for \textit{rpCholesy-QR}  requires  only a minimal amount of assumptions and produces
true bounds, rather than first-order estimates.
\end{enumerate}

\subsection{Existing work}
We distinguish between deterministic and randomized Cho\-lesky-QR algorithms.

\paragraph{Deterministic algorithms}
While the computation $\mq=\ma\mr^{-1}$ was pointed out in \cite[section 1]{Bjorck1967},
 the notion of a Cholesky-QR algorithm can apparently be traced back to Poincar\'{e} and
chemistry literature from the 1950's \cite[section 1]{StathWu2002}. 

Recent work is concerned with reducing the dependence of the orthogonalization error, that 
is, the deviation of the computed $\mq$ from orthonormality,
on the squared condition number of $\ma$.
Stathopoulos and Wu \cite{StathWu2002} present an approach that replaces the Cholesky 
factorization with a regularized  SVD of $\ma$.
A mixed precision Cholesky QR \cite{Yama2015} for use on GPUs computes the first two steps (Gram matrix formation and Cholesky factorization) in higher arithmetic
precision. 

Akin to the 'twice is enough' orthogonalization strategy, the \textit{Cholesky-QR2} algorithm
\cite{YNYF2015} repeats the Cholesky QR algorithm on the computed orthonormal
factor. However, the algorithm can still break down
if $\ma$ is too ill-conditioned for the Gram matrix to have a Cholesky factorization.
The remedy in \cite{FKNYY2020} is to run Cholesky-QR thrice,
with the first stage ensuring the existence of a Cholesky factorization by shifting the Gram matrix prior to the factorization to ensure positive definiteness. 

Conditions in \cite{Barlow2024} ensure the numerical backward stability of
 a block Gram-Schmidt algorithm based on mixed-precision Cholesky-QR algorithms
from~\cite{Yama2015}. The orthogonalization error of a randomized block Gram-Schmidt algorithm 
can be improved by post-processing it with Cholesky-QR \cite[Remark 2.1]{BG2023}.

Cholesky-QR algorithms in oblique inner products are considered in \cite{YNYF2016},
and Cholesky-LU-QR algorithms in \cite{TOO2020}.

\paragraph{Randomized algorithms}
Current algorithms are based on sketching with Gaussian matrices or sampling
without prior smoothing.
A sketched Cholesky-QR in a sketch-orthogonal basis 
is mentioned in \cite[Remark 2.10]{BG2022} and presented in  \cite[Algorithm 2]{Balabanov2022} and
\cite[Algorithm 2.2]{BG2023},
while a preconditioned sketched Cholesky-QR algorithm is presented in
\cite[Algorithm 3]{Balabanov2022} and \cite[Algorithm 2.5]{BG2023}. Cholesky-QR algorithms with preconditioners from randomized 
LU and QR factorizations
in \cite{FanGuo2021} are based on multiplication by Gaussians or row sampling, but without
prior smoothing to improve the coherence.
A multi-sketch algorithm is presented in \cite{HSBY2023}.
Rank-revealing randomized Cholesky-QR algorithms with pivoting are proposed in \cite{Balabanov2022,MBMDML24}.

\subsection{Overview}
To set the stage for the preconditioned algorithm, we present the motivation
for and an analysis of the basic Cholesky-QR algorithm (Section~\ref{s_basic}).
These are followed by the perturbation analysis of
a preconditioned Cholesky-QR algorithm with a fixed user-specified preconditioner 
(Section~\ref{s_precond}) and our randomized preconditioned Cholesky-QR algorithm
\textit{rpCholesky-QR} (Section~\ref{s_rand}). Numerical experiments conclude the paper
(Section~\ref{s_exp}).

\subsection{Notation}
 The singular values of a matrix $\ma\in\rmn$ with $m\geq n$ are 
 $\sigma_1(\ma)\geq \cdots\geq \sigma_n(\ma)\geq 0$.
 The two-norm condition number with respect to left inversion of $\ma\in\rmn$
 with $\rank(\ma)=n$ is 
 $\kappa(\ma)\equiv \|\ma\|_2\|\ma^{\dagger}\|_2=\sigma_1(\ma)/\sigma_n(\ma)$.
 The eigenvalues of a symmetric matrix $\mg\in\rnn$ are 
 $\lambda_1(\mg)\geq \cdots \geq \lambda_n(\mg)$.
 The columns of the identity matrix are 
 $\mi_n=\begin{bmatrix}\ve_1 & \cdots & \ve_n\end{bmatrix}\in\rnn$.

  \section{Basic Cholesky QR}\label{s_basic}
  To set the stage for the preconditioned algorithm, we present the motivation
for the basic Cholesky-QR algorithm in Section~\ref{s_bcqr}, and a perturbation analysis
in Section~\ref{s_bcqrs}.

 \subsection{Basic Cholesky-QR in exact arithmetic}\label{s_bcqr}
 Given a tall and skinny matrix $\ma\in\rmn$ with $\rank(\ma)=n$ and $m\gg n$,
 the goal is to compute a thin
QR decomposition 
\begin{align*}
\ma=\mq\mr
\end{align*}
 where $\mq\in\rmn$ has orthonormal columns, $\mq^T\mq=\mi_n$, and $\mr\in\rnn$
 is upper triangular nonsingular.
 
 Algorithm~\ref{alg_0} reduces the dimension of the problem by computing the Cholesky decomposition of the smaller dimensional Gram matrix $\mg\equiv\ma^T\ma\in\rnn$, which yields the
 upper triangular matrix $\mr$. A subsequent multiplication of $\ma$ with $\mr^{-1}$ produces the orthonormal factor $\mq$.
 The multiplication with $\mr^{-1}$ is implemented as $m$ lower triangular solves $\mr^T\mq^T=\ma^T$, one for each row of~$\mq$.

\begin{algorithm}
\caption{Basic Cholesky-QR}\label{alg_0}
\begin{algorithmic}[1]
\REQUIRE $\ma\in\rmn$ with $\rank(\ma)=n$ 
\ENSURE Thin QR decomposition $\ma=\mq\mr$
\medskip

\STATE{Multiply $\mg=\ma^T\ma$} \qquad \COMMENT{Gram matrix}
\STATE{Factor $\mg=\mr^T\mr$} \qquad \COMMENT{Triangular Cholesky factor $\mr\in\rnn$  of Gram  matrix}
\STATE{Solve $\mq=\ma\mr^{-1}$} \qquad 
\COMMENT{Orthonormal QR factor  of $\ma$}
\end{algorithmic}
\end{algorithm}

The idea behind the Cholesky-QR algorithm is that any two `Cholesky' factors of a  matrix
are orthogonally related. 

\begin{lemma}\label{l_r1r2}
Let $\mg\in\rnn$ be symmetric positive definite, with factorizations
\begin{align*}
\mg=\mr_1^T\mr_1=\mr_2^T\mr_2,
\end{align*}
where $\mr_1\in\rnn$ is nonsingular, and $\mr_2\in\rmn$ has $\rank(\mr_2)=n$.

Then $\mr_2=\mq\mr_1$ where $\mq\equiv \mr_2\mr_1^{-1}$ satisfies $\mq^T\mq=\mi_n$.

If $\mr_1$ and $\mr_2$ are both square with positive diagonal elements, then $\mq=\mi$,
confirming the uniqueness of the Cholesky factorization.
\end{lemma}

Lemma~\ref{l_r1r2} is the full-rank case of \cite[Proposition 4]{Crone81} with an added explicit
expression for~$\mq$.

\subsection{Perturbation analysis of basic Cholesky-QR}\label{s_bcqrs}  
We analyze the sensitivity of the basic Cholesky-QR Algorithm~\ref{alg_0} 
  with the model in Algorithm~\ref{alg_0s}, where the perturbations are numbered according to the steps in which they occur: 
  
\begin{description}
\item[$\me$:\ ] input perturbation of $\ma$;
\item[$\me_1$:\ ] forward error in the multiplication of $\ma+\me$ with its transpose;
\item[$\me_2$:\ ] backward error in the Cholesky factorization of $\widehat{\mg}$; 
\item[$\me_3$:\ ] backward error (residual) in the solution of the linear system with matrix $\widehat{\mr}$
and right-hand side $\ma+\me$.
\end{description}
     
\begin{algorithm}
\caption{Perturbed Basic Cholesky-QR}\label{alg_0s}
\begin{algorithmic}[1]
\REQUIRE $\ma\in\rmn$  with $\rank(\ma)=n$
\ENSURE Thin QR decomposition $\ma\approx\widehat{\mq}\widehat{\mr}$
\medskip

\STATE{Multiply $\widehat{\mg}=(\ma+\me)^T(\ma+\me)+\me_1$} 
\STATE{Factor $\widehat{\mg}+\me_2=\widehat{\mr}^T\widehat{\mr}$} \qquad 
\STATE{Solve $\widehat{\mq}=\left((\ma+\me)+\me_3\right)\widehat{\mr}^{-1}$} 
\end{algorithmic}
\end{algorithm}

Theorem~\ref{t_perturb1} below presents an analysis of Algorithm~\ref{alg_0s}, which is the perturbed version of the exact Algorithm~\ref{alg_0}. In IEEE double precision, the norm wise relative input error can be expected to be 
 $\epsilon_A=\|\me\|_2/\|\ma\|_2\approx 10^{-16}$.

  \begin{theorem}\label{t_perturb1}
 Let $\ma\in\rmn$ with $\rank(\ma)=n$. 
Assume that the errors $\me_1, \me_2\in\rnn$ in Algorithm~\ref{alg_0s}
are symmetric, and
  \begin{align}
 \widehat{\mg}&=(\ma+\me)^T(\ma+\me) +\me_1, \qquad \epsilon_A\equiv \frac{\|\me\|_2}{\|\ma\|_2}, \quad \epsilon_1\equiv\frac{\|\me_1\|_2}{\|\ma+\me\|_2^2}
 \label{s_12}\\
 \widehat{\mg}+\me_2&=\widehat{\mr}^T\widehat{\mr}, \qquad 
 \epsilon_2\equiv \frac{\|\me_2\|_2}{\|\widehat{\mg}\|_2}
\label{s_22}\\
 \widehat{\mq}\widehat{\mr}&=(\ma+\me)+\me_3, \qquad 
 \epsilon_3\equiv \frac{\|\me_3\|_2}{\|\ma+\me\|_2}.\label{s_32}
 \end{align}
 Define
 \begin{align*}
 \gamma_1&\equiv (1+\epsilon_A)^2
\left( \epsilon_1+(1+\epsilon_1)\epsilon_2+2\epsilon_3+\epsilon_3^2\right)\\
\gamma_2&\equiv 2\epsilon_A+\epsilon_A^2+(1+\epsilon_A)^2\left(\epsilon_1+(1+\epsilon_1)\epsilon_2\right).
  \end{align*} 
If $\kappa(\ma)^2\gamma_2<1$, then $\widehat{\mg}+\me_2$ is symmetric positive definite, and
   \begin{align}
   \kappa(\widehat{\mr})&\leq \kappa(\mr)\>
   \sqrt{\frac{1+\gamma_2}{1-\kappa(\ma)^2\gamma_2}}\\
    \|\mi-\widehat{\mq}^T\widehat{\mq}\|_2&\leq  
 \frac{\kappa(\ma)^2\>\gamma_1}{1-\kappa(\ma)^2\>\gamma_2}\label{e_perturb1}\\
 \frac{\|\ma-\widehat{\mq}\widehat{\mr}\|_2}{\|\ma\|_2}&\leq \epsilon_A+(1+\epsilon_A)\>
 \epsilon_3.\label{e_perturb2}
 \end{align}
 \end{theorem}

 \begin{proof}
 This is a special case of Theorem~\ref{t_perturb2} where
 $\me_s=\vzero$, $\mr_s=\mi$ and $\me_4=\vzero$, resulting in
 $\epsilon_F=\epsilon_A$, $\epsilon_4=0$ and $\eta=1$.
 \end{proof}

The errors in the deviation of $\widehat{\mq}$ from orthonormality (\ref{e_perturb1})
 are amplified by the square of the condition number of $\ma$.
  Like most residuals from linear solvers, the residual (\ref{e_perturb2}) of the computed 
  QR factorization shows no dependence on the condition number of the coefficient matrix $\widehat{\mr}$.

\section{Preconditioned Cholesky-QR}\label{s_precond}
We present a preconditioned Cholesky-QR algorithm with a fixed user-specified preconditioner in
Section~\ref{s_pcqr}, perturbation results in Section~\ref{s_pcqrs}, and the derivation in
Section~\ref{s_proof2}.

\subsection{Preconditioned Cholesky-QR in exact arithmetic}\label{s_pcqr}
Instead of computing the Gram matrix of the original matrix~$\ma$, Algorithm~\ref{alg_1}
first preconditions~$\ma$ with the nonsingular matrix $\mr_s$, and then applies
the Cholesky-QR algorithm to
the hopefully better conditioned matrix $\ma_1\equiv\ma\mr_s^{-1}$.

\begin{algorithm}
\caption{Preconditioned Cholesky-QR}\label{alg_1}
\begin{algorithmic}[1]
\REQUIRE $\ma\in\rmn$ with $\rank(\ma)=n$, nonsingular preconditioner $\mr_s\in\rnn$
\ENSURE Thin QR decomposition $\ma=\mq\mr$
\medskip

\STATE \COMMENT{Precondition}
\STATE{Solve $\ma_1=\ma\mr_s^{-1}$} \qquad \COMMENT{$\ma_1\in\rmn$ is preconditioned version of $\ma$}
\medskip

\STATE \COMMENT{Cholesky-QR of $\ma_1$}
\STATE{Multiply $\mg_1=\ma_1^T\ma_1$} \qquad \COMMENT{Gram matrix}
\STATE{Factor $\mg_1=\mr_2^T\mr_2$} \qquad \COMMENT{Cholesky factor $\mr_2\in\rnn$  of Gram  matrix}
\STATE{Solve $\mq=\ma_1\mr_2^{-1}$} \qquad 
\COMMENT{Orthonormal QR factor of $\ma_1$ and $\ma$}
\medskip

\STATE \COMMENT{Recover $\mr$}
\STATE{Multiply $\mr=\mr_2\mr_s$} \qquad \COMMENT{Triangular QR factor of $\ma$}
\end{algorithmic}
\end{algorithm}
 
 \begin{remark}\label{r_precond}
As long as $\mr_s$ is nonsingular, the preconditioned matrix $\ma_1$ 
has the same orthonormal factor as the original matrix $\ma$.

This is because, if  $\mr_s$ is nonsingular, then so is $\mr_2$. Then lines 2, 6 and~8 of Algorithm~\ref{alg_1} imply
\begin{align}\label{e_q}
\ma=\ma_1\mr_s=\mq\mr_2\mr_s.
\end{align}
\end{remark}

 \subsection{Perturbation analysis of preconditioned Cholesky-QR}\label{s_pcqrs}
We analyze the sensitivity of Algorithm~\ref{alg_1} 
  with the model in Algorithm~\ref{alg_1s}.
 The perturbations are numbered as in  
 the perturbed  basic Cholesky-QR Algorithm~\ref{alg_1s},
  \begin{description}
\item[$\me$:\ ] input perturbation of $\ma$;
\item[$\me_s$:\ ] backward error (residual) in the solution of the linear system with
matrix $\mr_s$ and solution $\widehat{\ma}_1$;
\item[$\me_1$:\ ] forward error in the multiplication of $\widehat{\ma}_1$ with its transpose;
\item[$\me_2$:\ ] backward error in the Cholesky factorization of $\widehat{\mg}_1$; 
\item[$\me_3$:\ ] backward error (residual) in the solution of the linear system
with matrix $\widehat{\mr}_2$ and right hand side $\widehat{\ma}_1$;
\item[$\me_4$:\ ] forward error in the multiplication of $\widehat{\mr}_2$ with $\mr_s$.
\end{description}

\begin{algorithm}
\caption{Perturbed Preconditioned Cholesky-QR}\label{alg_1s}
\begin{algorithmic}[1]
\REQUIRE $\ma\in\rmn$ with $\rank(\ma)=n$, nonsingular preconditioner $\mr_s\in\rnn$ 
\ENSURE Thin QR decomposition $\ma\approx\widehat{\mq}\widehat{\mr}$
\medskip

\STATE \COMMENT{Precondition}
\STATE{Solve $\widehat{\ma}_1=\left((\ma+\me)+\me_s\right)\mr_s^{-1}$} 
\medskip

\STATE \COMMENT{Cholesky-QR of $\widehat{\ma}_1$}
\STATE{Multiply $\widehat{\mg}_1=\widehat{\ma}_1^T\widehat{\ma}_1+\me_1$} 
\STATE{Factor $\widehat{\mg}_1+\me_2=\widehat{\mr}_2^T\widehat{\mr}_2$} 
\STATE{Solve $\widehat{\mq}=(\widehat{\ma}_1+\me_3)\widehat{\mr}_2^{-1}$} 
\medskip

\STATE \COMMENT{Recover $\widehat{\mr}$}
\STATE{Multiply $\widehat{\mr}=\widehat{\mr}_2\mr_s+\me_4$} 
\end{algorithmic}
\end{algorithm}
 
 Theorem~\ref{t_perturb2} below presents
 an analysis of Algorithm~\ref{alg_1s}, which is the perturbed version of the exact Algorithm~\ref{alg_1}.

 \begin{theorem}\label{t_perturb2}
 Let $\ma\in\rmn$ with $\rank(\ma)=n$, and let $\mr_s\in\rnn$ be nonsingular.
Assume that  the errors $\me_1, \me_2\in\rnn$ in Algorithm~\ref{alg_1s} are symmetric and
  \begin{align}
 \widehat{\ma}_1&=\left((\ma+\me)+\me_s\right)\mr_s^{-1}, \qquad 
 \epsilon_A\equiv\frac{\|\me\|_2}{\|\ma\|_2}, \qquad
 \epsilon_s\equiv\frac{\|\me_s\|_2}{\|\ma_1\|_2\|\mr_s\|_2}\label{s_02a}\\
  \widehat{\mg}_1&=\widehat{\ma}_1^T\widehat{\ma}_1 +\me_1, 
 \qquad \epsilon_1\equiv \frac{\|\me_1\|_2}{\|\widehat{\ma}_1\|_2^2}, 
 \label{s_12a}\\
 \widehat{\mg}_1+\me_2&=\widehat{\mr}_2^T\widehat{\mr}_2, \qquad 
 \epsilon_2\equiv \frac{\|\me_2\|_2}{\|\widehat{\mg}_1\|_2}
\label{s_22a}\\
\widehat{\mq} \widehat{\mr}_2&=\widehat{\ma}_1+\me_3, \qquad 
 \epsilon_3\equiv \frac{\|\me_3\|_2}{\|\widehat{\ma}_1\|_2}\label{s_32a}\\
 \widehat{\mr}&=\widehat{\mr}_2\mr_s+\me_4,\qquad 
 \epsilon_4\equiv \frac{\|\me_4\|}{\|\widehat{\mr}_2\mr_s\|_2}.
\label{s_42a}
 \end{align}
 Define
 \begin{align*}
 \eta&\equiv \frac{\|\ma_1\|_2\|\mr_s\|_2}{\|\ma\|_2}\qquad \text{where}\qquad
 1\leq \eta\leq \kappa(\ma_1)\\
 \epsilon_F&\equiv (\epsilon_A+\epsilon_s)\kappa_2(\mr_s)\\ 
  \gamma_1&\equiv (1+\epsilon_F)^2
( \epsilon_1+(1+\epsilon_1)\epsilon_2+2\epsilon_3+\epsilon_3^2)\\
\gamma_2&\equiv 2\epsilon_F+\epsilon_F^2+(1+\epsilon_F)^2\left(\epsilon_1+
(1+\epsilon_1)\epsilon_2\right)\\
\gamma_3&\equiv \epsilon_4(1+\epsilon_F)(1+\epsilon_3).
\end{align*} 
If $\kappa(\ma_1)^2\gamma_2<1$, then $\widehat{\mg}_1+\me_2$ is symmetric positive definite, and
   \begin{align}
   \kappa(\widehat{\mr}_2)&\leq \kappa(\mr_2)\>
   \sqrt{\frac{1+\gamma_2}{1-\kappa(\ma_1)^2\gamma_2}}\label{e_perturb00a}\\
 \|\mi-\widehat{\mq}^T\widehat{\mq}\|_2&\leq 
 \frac{\kappa(\ma_1)^2\gamma_1}{1-\kappa(\ma_1)^2\gamma_2}\label{e_perturb01a}\\
\frac{\|\ma-\widehat{\mq}\widehat{\mr}\|_2}{\|\ma\|_2}&\leq
\epsilon_A+\left(\epsilon_s+(1+\epsilon_F)\epsilon_3\right)\eta +{\gamma_3}
\sqrt{\frac{1+\gamma_2}{1-\kappa(\ma_1)^2\gamma_2}}\>\eta\>\kappa(\ma_1).
    \label{e_perturb2a}
 \end{align}
 \end{theorem} 

\begin{proof} 
See section~\ref{s_proof2}.
\end{proof}

The bounds in Theorem~\ref{t_perturb2} have the following properties.

\begin{enumerate}
\item The quantity 
\begin{align*}
\eta\equiv \frac{\|\ma_1\|_2\|\mr_s\|_2}{\|\ma\|_2}=\frac{\|\ma_1\|_2\|\mr_s\|_2}{\|\ma_1\mr_s\|_2}
\end{align*}
is the condition number for the multiplication of $\ma_1$ and $\mr_s$ \cite[Fact 2.22]{IIbook}.
\item Unlike \cite[sections 4, Appendix~A]{HSBY2023}, there is no dependence on the condition number of 
the original matrix $\ma$.
\item All bounds depend on the condition number $\kappa(\ma_1)$ 
of the preconditioned matrix and on
the condition number of the preconditioner $\mr_s$, via
$\epsilon_F=(\epsilon_A+\epsilon_s)\kappa(\mr_s)$.

\item Unlike \cite[section~3.2]{YNYF2015} which presents the deviation from orthonormality in the 
Frobenius norm, the bounds in Theorem~\ref{t_perturb2} are in the two-norm. 
\end{enumerate}

\begin{corollary}[First-order version of Theorem~\ref{t_perturb2}]\label{c_perturb2}
Under the assumptions from Theorem~\ref{t_perturb2},
the respective first order terms corresponding to $\gamma_1$, $\gamma_2$ and $\gamma_3$
are
\begin{align*}
\widetilde{\gamma}_1&\equiv \epsilon_1+\epsilon_2+2\epsilon_3\\
\widetilde{\gamma}_2& \equiv 2\epsilon_F+\epsilon_1+\epsilon_2
=2(\epsilon_A+\epsilon_s)\kappa_2(\mr_s)+\epsilon_1+\epsilon_2\\
\widetilde{\gamma}_3&\equiv \epsilon_4,
\end{align*}
and the first order versions of the bounds are
 \begin{align*}
   \kappa(\widehat{\mr}_2)&\lesssim \sqrt{1+\widetilde{\gamma}_2(1+\kappa(\ma_1)^2)}\>
   \kappa(\mr_2)\\
  \|\mi-\widehat{\mq}^T\widehat{\mq}\|_2&\lesssim  
  (\epsilon_1+\epsilon_2+2\epsilon_3)\>\kappa(\ma_1)^2\\
\frac{\|\ma-\widehat{\mq}\widehat{\mr}\|_2}{\|\ma\|_2}&\lesssim
 \epsilon_A + \left(\epsilon_s+\epsilon_3
 +\epsilon_4\>\kappa(\ma_1)\right)\eta.
  \end{align*}
\end{corollary}

The bounds in Corollary~\ref{c_perturb2} have the following properties.
\begin{enumerate}
\item To first order, the deviation of $\widetilde{\mq}$ from orthonormality and the residual
do not depend on the condition number $\kappa(\mr_s)$ of the preconditioner $\mr_s$.
  \item
 To first order, $\|\mi-\widehat{\mq}^T\widehat{\mq}\|_2$ does not depend on $\epsilon_A$
 and $\epsilon_s$. Thus, the deviation of $\widehat{\mq}$ from orthonormality
 does not depend strongly on the accuracy of the computed $\widehat{\ma}_1$.
 This makes sense in light of  Remark~\ref{r_precond}.
 
 However, the experiments in Section~\ref{s_exp2} illustrate that the deviation from 
orthonormality tends to grow only with $\kappa(\ma_1)$, rather than $\kappa(\ma_1)^2$.
  
  \item To first order, $\|\ma-\widehat{\mq}\widehat{\mr}\|_2/\|\ma\|_2$ does not depend on 
  $\epsilon_1$ and $\epsilon_2$. Thus, the residual does not depend strongly
  on the accuracy of the Gram matrix formation and the Cholesky decomposition.
It does depend, though, on the accuracy of the final triangular factor.  
 
 However, the numerical experiments in Section \ref{s_exp1} and~\ref{s_exp3} indicate that 
 the residual  remains at
machine precision, hence this bound can be pessimistic.
\end{enumerate}

\subsection{Proof of Theorem~\ref{t_perturb2}}\label{s_proof2}
The proof consists of eight steps.

\begin{enumerate}
\item Bound the forward error in $\widehat{\ma}_1$,
\begin{align}\label{e_aux32}
\widehat{\ma}_1=\left((\ma+\me)+\me_s\right)\mr_s^{-1}=
\underbrace{\ma\mr_s^{-1}}_{\ma_1}+\underbrace{(\me+\me_s)\mr_s^{-1}}_{\mf_s}
=\ma_1+\mf_s.
\end{align}
The triangle inequality implies 
\begin{align}\label{e_aux31}
\frac{\|\mf_s\|_2}{\|\ma_1\|_2}\leq 
\left(\frac{\|\me\|_2}{\|\ma_1\|_2}+\frac{\|\me_s\|_2}{\|\ma_1\|_2}\right)\|\mr_s^{-1}\|_2.
\end{align}
Line 4 of Algorithm~\ref{alg_1} implies
\begin{align*}
\frac{\|\me\|_2}{\|\ma_1\|_2}=\underbrace{\frac{\|\me\|_2}{\|\ma\|_2}}_{\epsilon_A}
\frac{\|\ma\|_2}{\|\ma_1\|_2}= \epsilon_A\frac{\|\ma_1\mr_s\|_2}{\|\ma_1\|_2}\leq
\epsilon_A\|\mr_s\|_2,
\end{align*}
while
\begin{align*}
\frac{\|\me_s\|_2}{\|\ma_1\|_2}=\underbrace{\frac{\|\me_s\|_2}{\|\mr_s\|_2\|\ma_1\|_2}}_{\epsilon_s}
\|\mr_s\|_2=\epsilon_s\|\mr_s\|_2.
\end{align*}
Substituting the previous two inequalities into (\ref{e_aux31}) gives
\begin{align}\label{e_aux34}
\frac{\|\mf_s\|_2}{\|\ma_1\|_2}\leq \epsilon_F\equiv (\epsilon_A+\epsilon_s)\kappa(\mr_s).
\end{align}
Furthermore, (\ref{e_aux32}) implies
\begin{align}\label{e_aux33}
\|\widehat{\ma}_1\|_2=\|\ma_1+\mf_s\|_2\leq \|\ma_1\|_2(1+\epsilon_F).
\end{align}

\item Bound $\|\widehat{\mg}_1\|_2$.
From (\ref{e_aux32}) and (\ref{s_12a}) follows
 \begin{align}\label{e_aux35}
 \widehat{\mg}_1&=(\ma_1+\mf_s)^T(\ma_1+\mf_s) +\me_1, \qquad 
 \epsilon_1\equiv\frac{\|\me_1\|_2}{\|\widehat{\ma}_1\|_2^2}.
\end{align}
Isolate the perturbations,
\begin{align}\label{e_aux35a}
\widehat{\mg}_1=\underbrace{\ma_1^T\ma_1}_{\mg_1}+
\underbrace{\mf_s^T\ma_1+\ma_1^T\mf_s+\mf_s^T\mf_s+\me_1}_{\mf_1}
\end{align}
where, with the abbreviation in (\ref{e_aux34}),
\begin{align*}
\|\mf_1\|_2\leq 2\|\ma_1\|_2\|\mf_s\|_2+\|\mf_s\|_2^2+\|\me_1\|_2=
\|\ma_1\|_2^2(2\epsilon_F+\epsilon_F^2)+\|\me_1\|_2.
\end{align*}
We bound the last summand with the help of (\ref{e_aux33}),
\begin{align*}
\|\me_1\|_2=\|\widehat{\ma}_1\|_2^2\>\frac{\|\me_1\|_2}{\|\widehat{\ma}_1\|_2^2}
\leq \|\ma_1\|_2^2\>(1+\epsilon_F)^2\epsilon_1.
\end{align*}
Substitute the above into the bound for $\|\mf_1\|_2$,
\begin{align}\label{e_aux37}
\|\mf_1\|_2\leq \|\ma_1\|_2^2\>(2\epsilon_F+\epsilon_F^2+(1+\epsilon_F)^2\epsilon_1)
\end{align}
and substitute this, in turn, into the norm of (\ref{e_aux35a}),
\begin{align}
\|\widehat{\mg}_1\|_2&\leq \|\ma_1\|_2^2+\|\mf_1\|_2
\leq \|\ma_1\|_2^2+\|\ma_1\|_2^2\>(2\epsilon_F+\epsilon_F^2+
(1+\epsilon_F)^2\epsilon_1)\nonumber\\
&=\|\ma_1\|_2^2\>(1+\epsilon_F)^2(1+\epsilon_1).\label{e_aux36}
\end{align}

\item Bound $\|\me_2\|_2$. From (\ref{s_12a}) and (\ref{e_aux36}) follows
\begin{align}\label{e_aux38}
\|\me_2\|_2&=\|\widehat{\mg}_1\|_2\>\underbrace{\frac{\|\me_2\|_2}{\|\widehat{\mg}_1\|_2}}_{\epsilon_2}
\leq \|\ma_1\|_2^2\left(1+\epsilon_F\right)^2(1+\epsilon_1)\epsilon_2.
\end{align}

\item Show that $\widehat{\mg}_1+\me_2$ is positive definite. From (\ref{e_aux35a}) follows
\begin{align*}
\widehat{\mg}_1+\me_2=\mg_1+\mf_1+\me_2.
\end{align*}
By assumption, $\me_1$ and $\me_2$ are symmetric. Hence, $\mf_1$ is symmetric
and  we can apply Weyl's theorem
as in step~4 in the proof of Theorem~\ref{t_perturb1},
\begin{align*}
|\lambda_j(\widehat{\mg}_1+\me_2)-\lambda_j(\mg_1)|\leq \|\mf_1+\me_2\|_2, \qquad 1\leq j\leq n.
\end{align*}
Bound the right-hand side with (\ref{e_aux37}) and (\ref{e_aux38}),
\begin{align*}
\|\mf_1+\me_2\|_2&\leq \|\mf_1\|_2+\|\me_2\|_2\\
&\leq \|\ma_1\|_2^2 \left(2\epsilon_F+\epsilon_F^2+(1+\epsilon_F)^2\epsilon_1+
\left(1+\epsilon_F\right)^2(1+\epsilon_1)\epsilon_2\right)\\
&=\|\ma_1\|_2^2\underbrace{\left(2\epsilon_F+\epsilon_F^2+(1+\epsilon_F)^2(\epsilon_1+
(1+\epsilon_1)\epsilon_2)\right)}_{\gamma_2}.
\end{align*}
Insert this into the above bound for Weyl's theorem,
 \begin{align}\label{e_weyla}
 |\lambda_j(\widehat{\mg}_1+\me_2)-\lambda_j(\mg_1)|\leq 
 \|\ma_1\|_2^2\>\gamma_2,\qquad 1\leq j\leq n.
 \end{align} 
In particular, if the smallest eigenvalue satisfies 
\begin{align}\label{e_ge1ba}
0<\lambda_n(\mg_1)-\|\ma_1\|_2^2\>\gamma_2\leq \lambda_n(\widehat{\mg}_1+\me_2),
\end{align}
 then $\widehat{\mg}_1+\me_2$ is nonsingular.
With  $\lambda_n(\mg_1)=\lambda_n(\ma_1^T\ma_1)=1/\|\ma_1^{\dagger}\|_2^2$, the first 
inequality in (\ref{e_ge1ba})
is equivalent to $\kappa_2(\ma_1)\gamma_2<1$, which
is true by assumption. 
Hence $\lambda_n(\widehat{\mg}_1+\me_2)>0$, and $\widehat{\mg}_1+\me_2$ is symmetric
positive definite.

\item Bound the condition number (\ref{e_perturb00a}).
 From (\ref{s_22a}) and  (\ref{e_weyla})  follows
  \begin{align*}
  \|\widehat{\mr}_2\|_2^2 &=\lambda_1(\widehat{\mg}_1+\me_2)
  \leq \lambda_1(\mg_1)+\|\ma_1\|_2^2\>\gamma_2=\|\ma_1\|_2^2\>(1+\gamma_2)
  \end{align*}
  and (\ref{e_ge1ba}) implies
    \begin{align}\label{e_rinva}
 \|\widehat{\mr}_2^{-1}\|_2^2&=\frac{1}{\lambda_n(\widehat{\mg}_1+\me_2)}\leq 
 \frac{1}{\lambda_n(\mg_1)-\|\ma_1\|_2^2\>\gamma_2}=
 \frac{\|\ma_1^{\dagger}\|_2^2}{1-\|\ma_1\|_2^2\|\ma_1^{\dagger}\|_2^2\>\gamma_2}.
 \end{align}
Combine the two bounds,
\begin{align*}
\kappa(\widehat{\mr}_2)^2\leq \kappa(\ma_1)^2\frac{1+\gamma_2}{1-\kappa(\ma_1)^2\>\gamma_2}.
\end{align*}

 \item Determine the  deviation of $\widehat{\mq}$ from orthonormality.
 Since $\widehat{\mg}_1+\me_2$ is nonsingular, so is $\widehat{\mr}_2$.
Substitute $\widehat{\mq}=(\widehat{\ma}_1+\me_3)\widehat{\mr}_2^{-1}$ from (\ref{s_32a}) into
  \begin{align*}
  \widehat{\mq}^T\widehat{\mq} &= \widehat{\mr}_2^{-T}(\widehat{\ma}_1+\me_3)^T
  (\widehat{\ma}_1+\me_3)\widehat{\mr}_2^{-1}\\
  &=\widehat{\mr}_2^{-T}(\widehat{\mg}_1+\mf_3)\widehat{\mr}_2^{-1}
 =\widehat{\mr}_2^{-T}(\widehat{\mg}_1+\me_2-\me_2+\mf_3)\widehat{\mr}_2^{-1}\\
  &=\mi+\widehat{\mr}_2^{-T}(\mf_3-\me_2)\widehat{\mr}_2^{-1},
  \end{align*}
  where the last equality follows from (\ref{s_22a}), and we abbreviate
  \begin{align*}
   \mf_3\equiv \widehat{\ma}_1^T\me_3+\me_3^T\widehat{\ma}_1+\me_3^T\me_3-\me_1.
   \end{align*}
  Therefore, the absolute deviation of $\widehat{\mq}$ from orthonormality is
  \begin{align*}
  \mi-\widehat{\mq}^T\widehat{\mq}=\widehat{\mr}_2^{-T}(\me_2-\mf_3)\widehat{\mr}_2^{-1}.
  \end{align*}

 \item Bound the deviation (\ref{e_perturb01a}) from orthonormality.
 The norm of the previous expression is bounded by
 \begin{align}\label{e_devoa}
 \| \mi-\widehat{\mq}^T\widehat{\mq}\|_2&\leq \|\widehat{\mr}_2^{-1}\|_2^2\>(\|\me_2\|_2+\|\mf_3\|_2).
 \end{align} 
  We bound $\|\mf_3\|_2$ with (\ref{s_12a}), (\ref{s_32a}) and (\ref{e_aux33}),
  \begin{align*}
  \| \mf_3\|_2 &= \|\widehat{\ma}_1^T\me_3+\me_3^T\widehat{\ma}_1+\me_3^T\me_3-\me_1\|_2\\
&    \leq 2\|\widehat{\ma}_1\|_2\|\me_3\|_2+\|\me_3\|_2^2 +\|\me_1\|_2\\
&  \leq 2\|\widehat{\ma}_1\|_2\|\me_3\|_2+\|\me_3\|_2^2 +\|\widehat{\ma}_1\|_2\>\epsilon_1\\
&\leq \|\widehat{\ma}_1\|_2^2\left(\epsilon_1+2\epsilon_3+\epsilon_3^2\right)\leq
   \|\ma_1\|_2^2\>(1+\epsilon_F)^2  \left(\epsilon_1+2\epsilon_3+\epsilon_3^2\right).
    \end{align*}
  Insert this bound for $\|\mf_3\|_2$ and (\ref{e_aux38}) into (\ref{e_devoa}),
 \begin{align*}
 \| \mi-\widehat{\mq}^T\widehat{\mq}\|_2&\leq \|\widehat{\mr}_2^{-1}\|_2^2\>
 (\|\mf_3\|_2+\|\me_2\|_2)  \\
&\leq \|\widehat{\mr}_2^{-1}\|_2^2\|\ma_1\|_2^2
\underbrace{(1+\epsilon_F)^2\left((1+\epsilon_1) \epsilon_2+\epsilon_1+
2\epsilon_3+\epsilon_3^2\right)}_{\gamma_1}.
   \end{align*}
Then substitute  (\ref{e_rinva}) into the above.

\item Bound the residual (\ref{e_perturb2a}).
From (\ref{s_32a})  and (\ref{s_42a}) follows 
\begin{align*}
\ma-\widehat{\mq}\widehat{\mr}&=
\ma-\widehat{\mq}\underbrace{(\widehat{\mr}_2\mr_s+\me_4)}_{\widehat{\mr}}=
\ma-\widehat{\mq}\widehat{\mr}_2\mr_s-\widehat{\mq}\me_4\\
&=\ma-\underbrace{(\widehat{\ma}_1+\me_3)\widehat{\mr}_2^{-1}}_{\widehat{\mq}}
\widehat{\mr}_2\mr_s -\widehat{\mq}\me_4
=\underbrace{\ma-(\widehat{\ma}_1+\me_3)\mr_s}_{\mz_1}-
\underbrace{\widehat{\mq}\me_4}_{\mz_2}=\mz_1-\mz_2.
\end{align*}
Taking norms and dividing by $\|\ma\|_2$ gives
\begin{align}\label{e_aux01}
\frac{\|\ma-\widehat{\mq}\widehat{\mr}\|_2}{\|\ma\|_2}\leq
\frac{\|\mz_1\|_2}{\|\ma\|_2}+\frac{\|\mz_2\|_2}{\|\ma\|_2}
\end{align}
It remains to bound $\|\mz_1\|_2$ and $\|\mz_2\|_2$.
Line 4 of Algorithm~\ref{alg_1}  implies
\begin{align*}\label{e_aux00}
\eta\equiv \frac{\|\ma_1\|_2\|\mr_s\|_2}{\|\ma\|_2}&= \frac{\|\ma_1\|_2\|\ma_1^{\dagger}\ma\|_2}{\|\ma\|_2} \leq \kappa(\ma_1)\\
&\geq \frac{\|\ma_1\mr_s\|_2}{\|\ma\|_2}=\frac{\|\ma\|_2}{\|\ma\|_2}=1.
\end{align*}

As for the first summand $\mz_1$ in (\ref{e_aux01}), we apply (\ref{e_aux32}), line~4 in Algorithm~\ref{alg_1}, and (\ref{e_aux32}) again
\begin{align*}
\mz_1&\equiv \ma-(\widehat{\ma}_1+\me_3)\mr_s=\ma-(\ma_1+\mf_s+\me_3)\mr_s
=-\mf_s\mr_s-\me_3\mr_s\\ 
&=-(\me+\me_s)-\me_3\mr_s.
\end{align*}
Take norms and apply the triangle inequality,
\begin{align}\label{e_mz1}
\frac{\|\mz_1\|_2}{\|\ma\|_2}\leq
\underbrace{\frac{\|\me\|_2}{\|\ma\|_2}}_{\epsilon_A}+\frac{\|\me_s\|}{\|\ma\|_2}+
\frac{\|\me_3\mr_s\|_2}{\|\ma\|_2}.
\end{align}
We bound each summand in turn.
For the second summand, (\ref{s_02a}) implies
\begin{align*}
\frac{\|\me_s\|_2}{\|\ma\|_2}&=\underbrace{\frac{\|\me_s\|_2}{\|\ma_1\|_2\|\mr_s\|_2}}_{\epsilon_s}
\underbrace{\frac{\|\ma_1\|_2\|\mr_s\|_2}{\|\ma\|_2}}_{\eta}\leq \epsilon_s \eta.
\end{align*}
For the third summand, (\ref{s_32a}) and (\ref{e_aux34})
 imply
\begin{align*}
\frac{\|\me_3\mr_s\|_2}{\|\ma\|_2}&\leq
\underbrace{\frac{\|\me_3\|_2}{\|\widehat{\ma}_1\|_2}}_{\epsilon_3}
\|\widehat{\ma}_1\|_2\frac{\|\mr_s\|}{\|\ma\|_2}
\leq \epsilon_3\>(1+\epsilon_F)\frac{\|\ma_1\|_2\|\mr_s\|_2}{\|\ma\|_2}
=\epsilon_3(1+\epsilon_F)\eta.
\end{align*}
Insert the above bounds into (\ref{e_mz1}),
\begin{align*}
\frac{\|\mz_1\|_2}{\|\ma\|_2}\leq\epsilon_A+\left(\epsilon_s+\epsilon_3(1+\epsilon_F)\right)\eta.
\end{align*}

As for the second summand $\mz_2\equiv\widehat{\mq}\me_4$ in (\ref{e_aux01}), we 
obtain from (\ref{s_42a})
\begin{align*}
\frac{\|\mz_2\|_2}{\|\ma\|_2}&\leq \frac{\|\widehat{\ma}_1+\me_3\|_2}{\|\ma\|_2}
\|\widehat{\mr}_2^{-1}\|_2
\underbrace{\frac{\|\me_4\|_2}{\|\widehat{\mr}_2\mr_s\|_2}}_{\epsilon_4}
\|\widehat{\mr}_2\mr_s\|_2\\
&\leq \epsilon_4\kappa(\widehat{\mr}_2)\>
\frac{\|\widehat{\ma}_1\|_2+\|\me_3\|_2}{\|\ma\|_2}\|\mr_s\|_2.
\end{align*}
From  (\ref{e_aux34}) and (\ref{s_32a})
follows that the product of the last two terms is bounded by
\begin{align*}
\frac{\|\widehat{\ma}_1\|_2+\|\me_3\|_2}{\|\ma\|_2}\|\mr_s\|_2&=
\frac{\|\ma_1\|_2\|\mr_s\|_2}{\|\ma\|_2}\frac{\|\widehat{\ma}\|_2+\|\me_3\|_2}{\|\ma_1\|_2}\\
&\leq \eta\left(1+\epsilon_F+\frac{\|\me_3\|_2}{\|\ma_1\|_2}\right)
\leq \eta
\left(1+\epsilon_F+\epsilon_3\>\frac{\|\widehat{\ma}_1\|_2}{\|\ma_1\|_2}\right)\\
&\leq \eta\left(1+\epsilon_F+\epsilon_3(1+\epsilon_F)\right)
= \eta(1+\epsilon_F)(1+\epsilon_3).
\end{align*}
Substitute this into the bound for $\|\mz_2\|_2/\|\ma\|_2$, and apply (\ref{e_perturb00a}),
\begin{align*}
\frac{\|\mz_2\|_2}{\|\ma\|_2}
&\leq \epsilon_4\>(1+\epsilon_3)(1+\epsilon_F)\>\eta\>\kappa(\widehat{\mr}_2)\\
&\leq\epsilon_4 (1+\epsilon_3)(1+\epsilon_F)
\sqrt{\frac{1+\gamma_2}{1-\kappa_2(\ma_1)^2\gamma_2}}\>\eta\>\kappa(\ma_1).
\end{align*}
At least substitute the bounds for $\|\mz_1\|_2/\|\ma\|_2$ and $\|\mz_2\|_2/\|\ma\|_2$
into (\ref{e_aux01}),
\begin{align*}
\frac{\|\ma-\widehat{\mq}\widehat{\mr}\|_2}{\|\ma\|_2}&\leq
\epsilon_A+\left(\epsilon_s+(1+\epsilon_F)\epsilon_3\right)\eta\\
&\qquad\qquad +\underbrace{\epsilon_4(1+\epsilon_F)(1+\epsilon_3)}_{\gamma_3}
\sqrt{\frac{1+\gamma_2}{1-\kappa_2(\ma_1)^2\gamma_2}}\>\eta\>\kappa(\ma_1).
\end{align*}$\quad$
\end{enumerate}

\section{Randomized preconditioned Cholesky-QR (\textit{rpCholesky-QR})}\label{s_rand}
We present  our randomized preconditioned Cholesky-QR algorithm, and a probabilistic
bound on the condition number of the preconditioned matrix.

The randomized preconditioner in line~2 of Algorithm~\ref{alg_2} is motivated by the least squares solver 
\textit{Blendenpik} \cite{Blendenpik}. The matrix $\mathcal{F}\in\rmm$ is a random orthogonal
matrix, and the matrix $\ms\in\real^{c\times m}$ samples rows of $\mathcal{F}\ma$
uniformly, independently and with replacement. Remark~\ref{r_exp}
shows the Matlab implementation.

\begin{algorithm}
\caption{Randomized preconditioned Cholesky-QR (\textit{rpCholesky-QR})}\label{alg_2}
\begin{algorithmic}[1]
\REQUIRE $\ma\in\rmn$ with $\rank(\ma)=n$ 
\ENSURE Thin QR decomposition $\ma=\mq\mr$
\medskip

\STATE \COMMENT{Randomized preconditioner}
\STATE{Compute $\ma_s= \ms\mathcal{F}\ma$} 
\COMMENT{Sample $c$ rows from smoothed matrix}
\STATE{Factor $\ma_s=\mq_s\mr_s$} \qquad \COMMENT{Thin QR decomposition of sampled matrix}
\STATE{Solve $\ma_1=\ma\mr_s^{-1}$} \qquad \COMMENT{$\ma_1\in\rmn$ is preconditioned version of $\ma$}
\medskip

\STATE \COMMENT{Cholesky-QR of $\ma_1$}
\STATE{Multiply $\mg_1=\ma_1^T\ma_1$} \qquad \COMMENT{Gram matrix}
\STATE{Factor $\mg_1=\mr_2^T\mr_2$} \qquad \COMMENT{Cholesky factor $\mr_2\in\rnn$  of Gram  matrix}
\STATE{Solve $\mq=\ma_1\mr_2^{-1}$} \qquad 
\COMMENT{Orthonormal QR factor of $\ma_1$ and $\ma$}
\medskip

\STATE \COMMENT{Recover $\mr$}
\STATE{Multiply $\mr=\mr_2\mr_s$} \qquad \COMMENT{Triangular QR factor of $\ma$}
\end{algorithmic}
\end{algorithm}
 
The condition number of the preconditioned matrix $\ma_1$
can be related to that of the sampled orthononormal matrix $\ms\mathcal{F}\mq$
  \cite{Blendenpik,IpsW12,RT08}.
The generalization below expresses the singular values of $\ms\mathcal{F}\mq$ in terms of 
the singular values of~$\ma_1$.

 \begin{lemma}\label{l_31} If $\rank(\ma_s)=n$ in Algorithm~\ref{alg_2}, then
\begin{align*}
\sigma_i(\ms\mathcal{F}\mq)=1/\sigma_{n-i+1}(\ma_1), \qquad 1\leq i\leq n,
\end{align*}
and $\kappa(\ms\mathcal{F}\mq)=\kappa(\ma_1)$.
\end{lemma}

\begin{proof}
From $\ma=\mq\mr$ and $\ms\mathcal{F}\ma=\mq_s\mr_s$ follows
\begin{align*}
\ms\mathcal{F}\mq\mr=\ms\mathcal{F}\ma =\ma_s=\mq_s\mr_s.
\end{align*}
Multiplying by $\mr^{-1}$  gives $\ms\mathcal{F}\mq=\mq_s\mr_s\mr^{-1}$, and the 
singular values 
\begin{align*}
\sigma_i(\ms\mathcal{F}\mq)&=\sigma_i(\mq_s\mr_s\mr^{-1})=\sigma_i(\mq_s\mr_s\mr^{-1}\mq^T)\\
&=\sigma_i(\mr_s\mr^{-1}\mq^T)=\sigma_i(\ma_1^{\dagger})=1/\sigma_{n-i+1}(\ma_1),
\qquad 1\leq i\leq n.
\end{align*}
\end{proof}

The following probabilistic bound on the rank and condition number of the preconditioned matrix
is a slight improvement over  \cite[Corollary 4.2]{IpsW12}.

\begin{theorem}\label{t_prob}
Let $\ma\in\rmn$ with $\rank(\ma)=n$ and QR factorization $\ma=\mq\mr$ where $\mq\in\rmn$
with $\mq^T\mq=\mi_n$. Let $\mathcal{F}\in\rmm$ be an orthogonal matrix, and 
let  $\mathcal{F}\mq$ 
have coherence $\mu\equiv \max_{1\leq i\leq m}{\|\ve_i^T\mathcal{F}\mq\|_2^2}$.
Let $\ms\in\real^{c\times n}$  sample $c$ rows uniformly, independently, and with replacement.

For any $0<\epsilon<1$ and $0<\delta<1$, if
\begin{align*}
c\geq 2m\,\mu\,\left(1+\frac{\epsilon}{3}\right)\,\frac{\ln{(n/\delta)}}{\epsilon^2}
\end{align*} 
then with probability at least $1-\delta$ we have  $\rank(\ma_1)=\rank(\ms\mathcal{F}\mq)=n$ and 
\begin{align*}
\kappa(\ma_1)=\kappa(\ms\mathcal{F}\mq)\leq \sqrt{\frac{1+\epsilon}{1-\epsilon}}.
\end{align*}
\end{theorem}

\begin{proof}
The structure of the proof is similar to that of \cite[Theorem 7.5]{HoI15}, and relies
on the concentration inequality \cite[Theorem 1.4]{Tropp12}.
\end{proof}

Unfortunately, the lower bound for the sampling amount in Theorem~\ref{t_prob}
is far too pessimistic. The numerical experiments in Section~\ref{s_exp} indicate that
$c=3n$ is often enough.


Our algorithm \textit{rpCholesky-QR} makes use of  the
following random orthogonal matrices, which result in almost optimal coherence.

\begin{remark}[Section 3.2 in \cite{Blendenpik}]
Let $\mathcal{F}=\mf_T\md$ where $\mf_T\in\rmm$ is a Walsh-Hadamard, discrete cosine, or Hartley transform. Let $\md\in\rmm$ be a diagonal matrix whose diagonal elements are independent
Rademacher variables, so that  $\md_{ii}=\pm 1$ with probability 1/2, $1\leq i\leq m$.
Then  with probability at least 0.95 the coherence~$\mu$ in Theorem~\ref{t_prob}
is bounded by
\begin{align*}
\mu\leq C\frac{n}{m} \ln{m}
\end{align*}
for some constant $C$.
\end{remark}

\section{Numerical Experiments}\label{s_exp}
We demonstrate the high accuracy of our algorithm \textit{rpCholesky-QR}
for very ill-conditioned matrices
and its superiority over \textit{Cholesky-QR2} \cite{YNYF2015},
another two-stage algorithm with about the same operation count.

 Section~\ref{s_exp1} illustrates the accuracy of \textit{rpCholesky-QR}
on numerically singular matrices.
Section~\ref{s_exp2} illustrates that the deviation
from orthonormality of \textit{rpCholesky} grows with only the condition number of the preconditioned
matrix, rather than its square. Section~\ref{s_exp3} 
compares \textit{rpCholesky-QR} to \textit{Cholesky-QR2}.

\begin{remark}\label{r_exp}
The randomized preconditioner in line~2 in Algorithm~\ref{alg_2} is computed with the Matlab 
commands
\smallskip

\begin{verbatim}
    D = spdiags(sign(rand(m, 1)-0.5), 0, m, m);
    FA = dct(D*A);
    Sampled_rows = randi(m, [c, 1]);
    A_s = sqrt(m/c)*FA(Sampled_rows, :);
\end{verbatim}
\end{remark}
\smallskip

\subsection{Accuracy of \textit{rpCholesky-QR} for numerically singular matrices}\label{s_exp1}
Figures \ref{fig_ErrorRange3}--\ref{fig_ErrorRangeG} illustrate the deviation 
from orthonormality, residual, and the condition number
of the preconditioned matrix over 10 trials for each sampling amount~$c$. 

We apply \textit{rpCholesky-QR} to numerically singular matrices $\ma\in\rmn$ with 
$\kappa(\ma)=10^{15}$ and
worst case coherence, 
\begin{align}\label{e_ErrorRange}
\ma=\mq_A\mr_A,
\end{align}
where $\mq_A$ and $\mr_A$ are computed with the Matlab commands
\smallskip

\begin{verbatim}
    Q_A = [eye(n); zeros(m-n,n)]; 
    R_A = gallery('randsvd', n, 10^15);
\end{verbatim}
\smallskip
These matrices do not necessarily satisfy the assumption $\kappa(\ma_1)^2\gamma_2<1$ 
in Theorem~\ref{t_perturb2}.

\begin{figure}\label{fig_ErrorRange3}
\begin{center}
\resizebox{3.1in}{!}
{\includegraphics{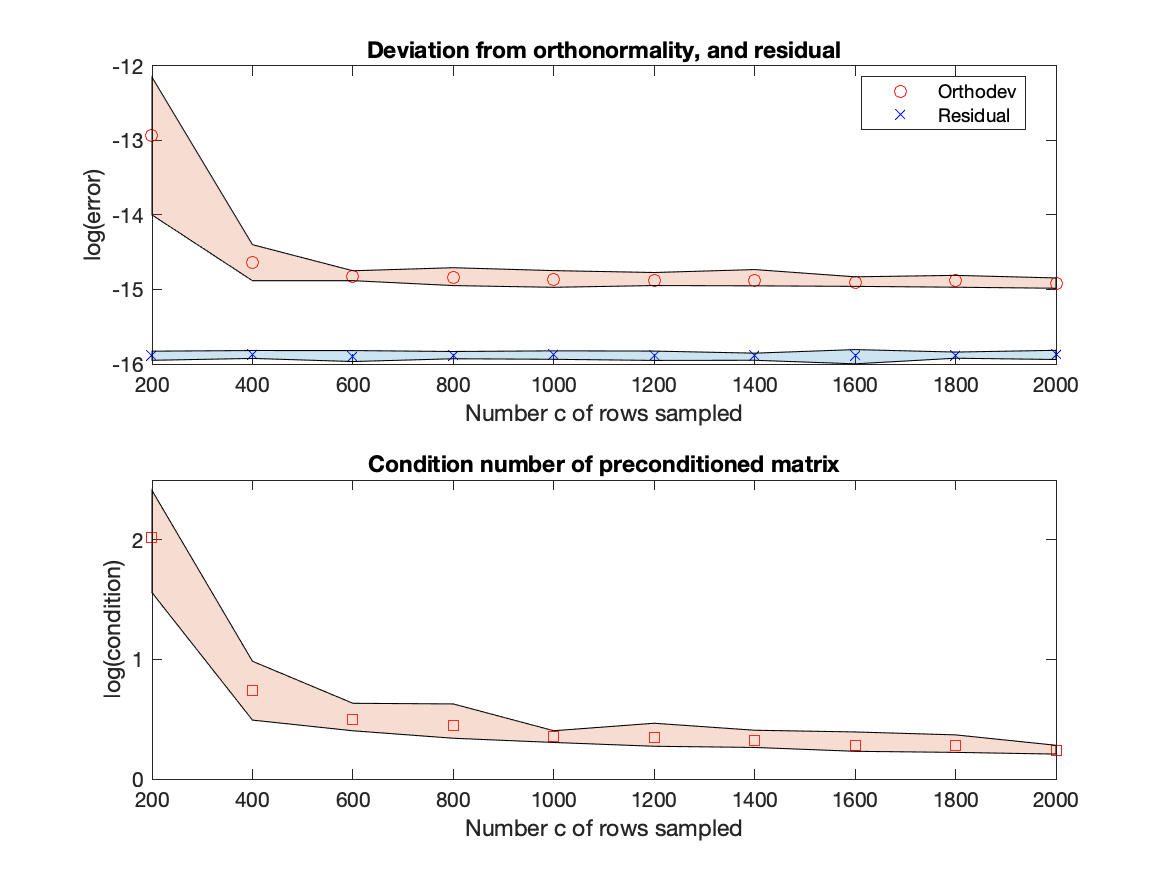}} 
\end{center}
\caption{Logarithm of errors and condition number versus sampling amount $c$ for
$\ma\in\real^{6,000\times 100}$ in (\ref{e_ErrorRange}).
Upper panel: The red region delineates
the smallest and largest deviation from orthonormality $\|\widehat{\mq}^T\widehat{\mq}-\mi\|_2$ 
over 10 runs for each sampling amount~$c$, and the red circles represent the mean.
The thin blue region delineates the smallest and largest residual
$\|\ma-\widehat{\mq}\widehat{\mr}\|_2/\|\ma\|_2$ over 10 runs for each sampling amount~$c$, 
and the blue crosses represent the mean.
Lower panel: The red region delineates the smallest and largest condition number
$\kappa_2(\ma_1)$ of the preconditioned matrix $\ma_1$ over 10 trials for each sampling 
amount~$c$, and the red squares represent the mean. }
\end{figure}

\paragraph{Figure~\ref{fig_ErrorRange3}}
The residual remains steadily at  $10^{-16}$ for all sampling
amounts. For very small sampling amounts $c=200=2n$, the deviation of $\widehat{\mq}$ from orthonormality is 
below $10^{-12}$. For sampling amounts $c\geq 600=6n$, the deviation of $\widehat{\mq}$ from orthonormality drops to about $10^{-15}$ and the condition number of
the preconditioned matrix to $\kappa(\ma_1)<10$.

\begin{figure}\label{fig_ErrorRange1}
\begin{center}
\resizebox{3.1in}{!}
{\includegraphics{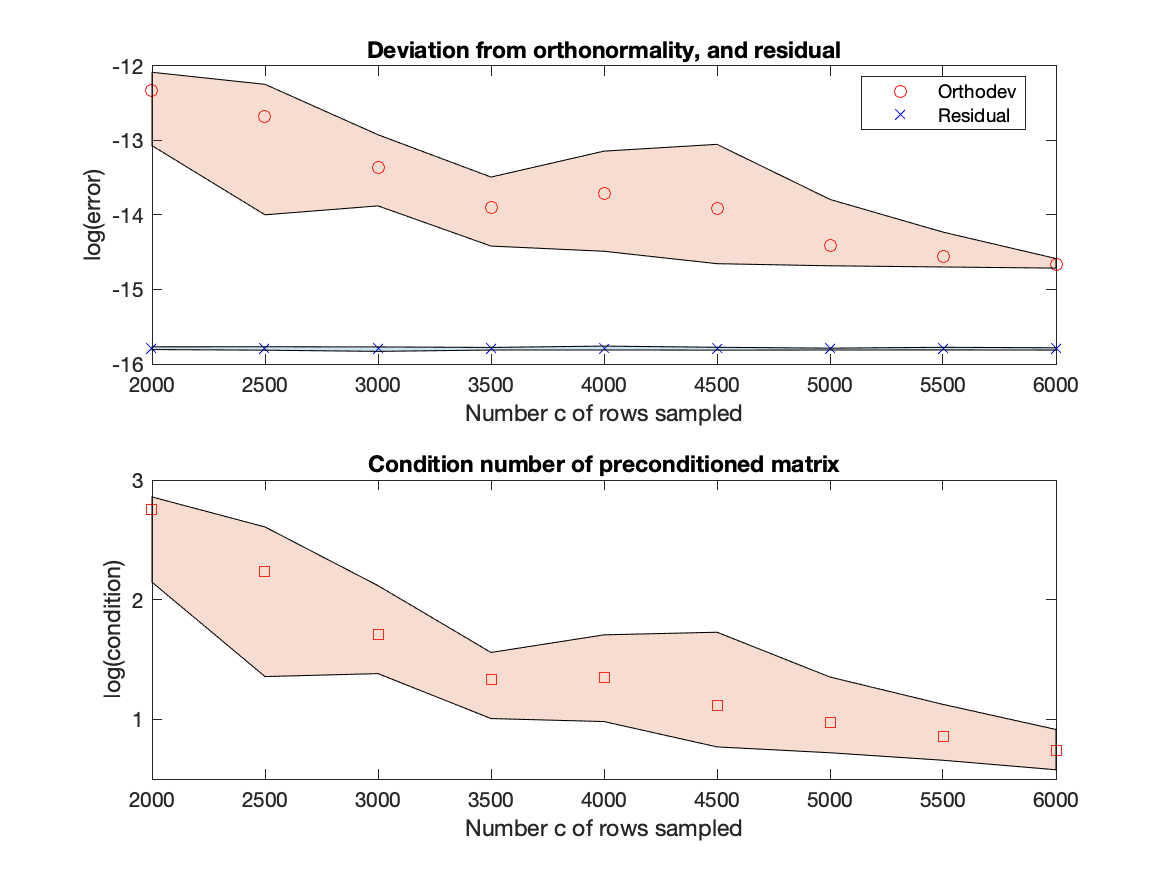}} 
\end{center}
\caption{Logarithm of errors and condition number versus sampling amount $c$ for
$\ma\in\real^{6,000\times 1,000}$ in (\ref{e_ErrorRange}).
Upper panel: The red region delineates
the smallest and largest deviation from orthonormality $\|\widehat{\mq}^T\widehat{\mq}-\mi\|_2$ 
over 10 runs for each sampling amount~$c$, and the red circles represent the mean.
The thin blue region delineates the smallest and largest residual
$\|\ma-\widehat{\mq}\widehat{\mr}\|_2/\|\ma\|_2$ over 10 runs for each sampling amount~$c$, 
and the blue crosses represent the mean.
Lower panel: The red region delineates the smallest and largest condition number
$\kappa_2(\ma_1)$ of the preconditioned matrix $\ma_1$ over 10 trials for each sampling 
amount~$c$, and the red squares represent the mean. }
\end{figure}

\paragraph{Figure~\ref{fig_ErrorRange1}}
The residual is close to $10^{-16}$ for all sampling
amounts. For sampling amounts $c\geq 3,000=3n$,
the deviation of $\widehat{\mq}$ from orthonormality starts to drop below $10^{-13}$,
and the condition number of
the preconditioned matrix to $\kappa(\ma_1)\leq 100$.

\begin{figure}\label{fig_ErrorRangeG}
\begin{center}
\resizebox{2.9in}{!}
{\includegraphics{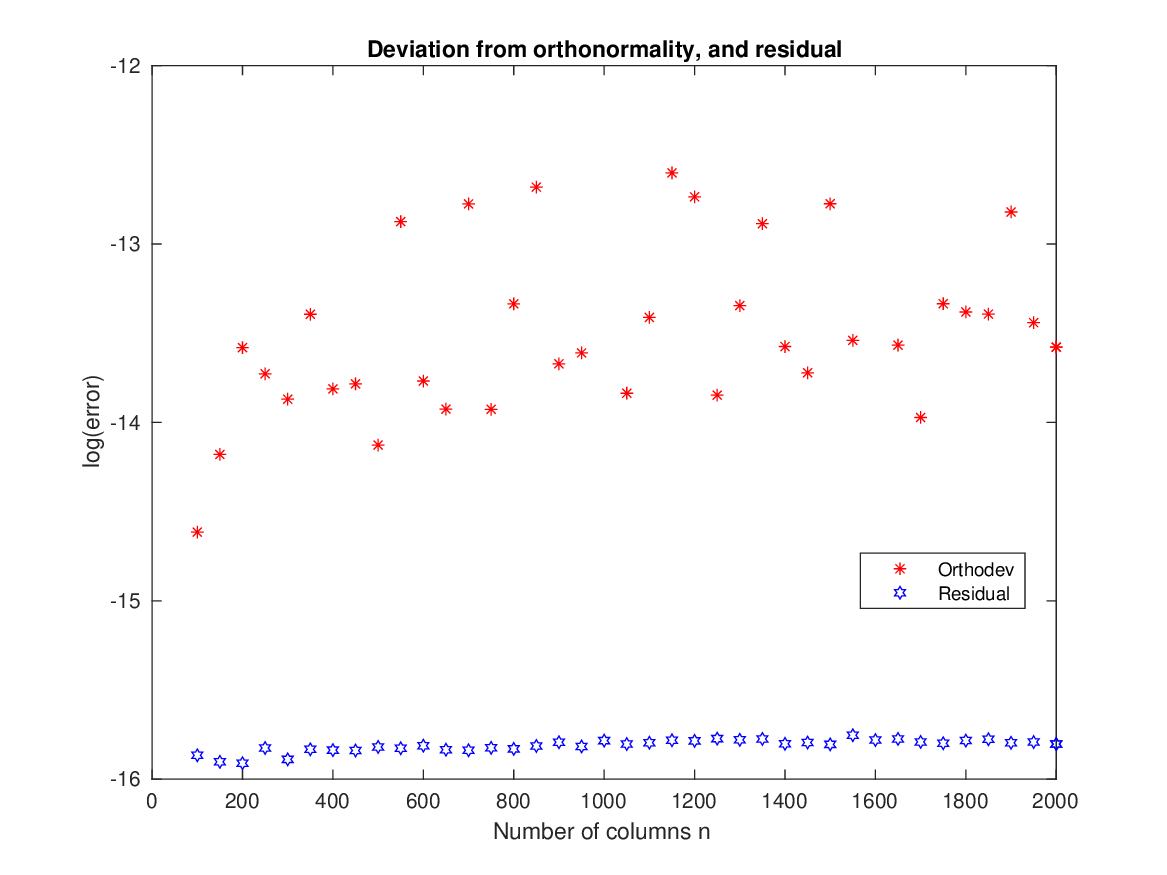}} 
\end{center}
\caption{Logarithm of deviation from orthonormality and residual versus number
of columns~$n$ for \textit{rpCholesky-QR} with $c=3n$ samples, 
applied to matrices
$\ma\in\real^{6,000\times n}$ in~(\ref{e_ErrorRange}) with $n=100, \ldots, 2,000$ columns.
The red stars represent the deviation from orthonormality
and the blue diamonds the residuals.
}
\end{figure}

\paragraph{Figure~\ref{fig_ErrorRangeG}}
We illustrate the deviation from orthonormality and residual versus number
of columns~$n$ for \textit{rpCholesky-QR} with $c=3n$ samples.
The algorithm is applied to matrices
$\ma\in\real^{6,000\times n}$ in (\ref{e_ErrorRange}) with a number of columns in the range 
$n=100, \ldots, 2,000$. 

With only a small number $3n$ of samples, \textit{rpCholesky-QR} 
produces residuals
slightly above $10^{-16}$ for all $n$, and a deviation from orthonormality below $10^{-12}$.

\paragraph{Summary}
Figures \ref{fig_ErrorRange3}--\ref{fig_ErrorRangeG} 
illustrate that even with small sampling amounts $c=3n$, \textit{rpCholesky-QR} produces
a deviation from orthonormality of at least $10^{-12}$ for numerically singular matrices~$\ma$.
The residuals remain below $10^{-15}$ and show hardly any variance, thus the residual bound
in Corollary~\ref{c_perturb2} is too pessimistic.

\subsection{Deviation from Orthonormality of \textit{rpCholesky-QR}}\label{s_exp2}
Figures \ref{fig_OrthoDev2}--\ref{fig_OrthoDevG}
 illustrate that \textit{rpCholesky-QR's} deviation from orthonormality 
 grows with the condition number of the preconditionded matrix, rather than its square,
and is captured by the simple estimate
 \begin{align}\label{e_ex1}
4\mathtt{eps} \>\kappa(\ma_1).
  \end{align}  
In contrast, setting
  $\epsilon_1=\epsilon_2=\epsilon_3=\mathtt{eps}=2.2\cdot 10^{-16}$
in Corollary~\ref{c_perturb2} gives the first order estimate
 \begin{align*}
\|\mi-\widehat{\mq}^T\widehat{\mq}\|_2&\lesssim  4\mathtt{eps} \>\kappa(\ma_1)^2.
  \end{align*}

  The matrices  $\ma\in\real^{6,000\times 100}$ 
and $\ma\in\real^{6,000\times 1,000}$ in (\ref{e_ErrorRange}) 
from Figures \ref{fig_OrthoDev2} and~\ref{fig_OrthoDev1}
are the same as in Figures \ref{fig_ErrorRange3} and \ref{fig_ErrorRange1}, respectively.
In both cases, (\ref{e_ex1}) 
estimates the correct magnitude of the deviation from orthonormality. 
Figure~\ref{fig_OrthoDevG} confirms this for matrices
$\ma\in\real^{6,000\times n}$  in (\ref{e_ErrorRange}) with a 
number of columns in the range $n=100, \ldots, 2,000$, and
small sampling amounts $c=3n$.

\begin{figure}\label{fig_OrthoDev2}
\begin{center}
\resizebox{2.9in}{!}
{\includegraphics{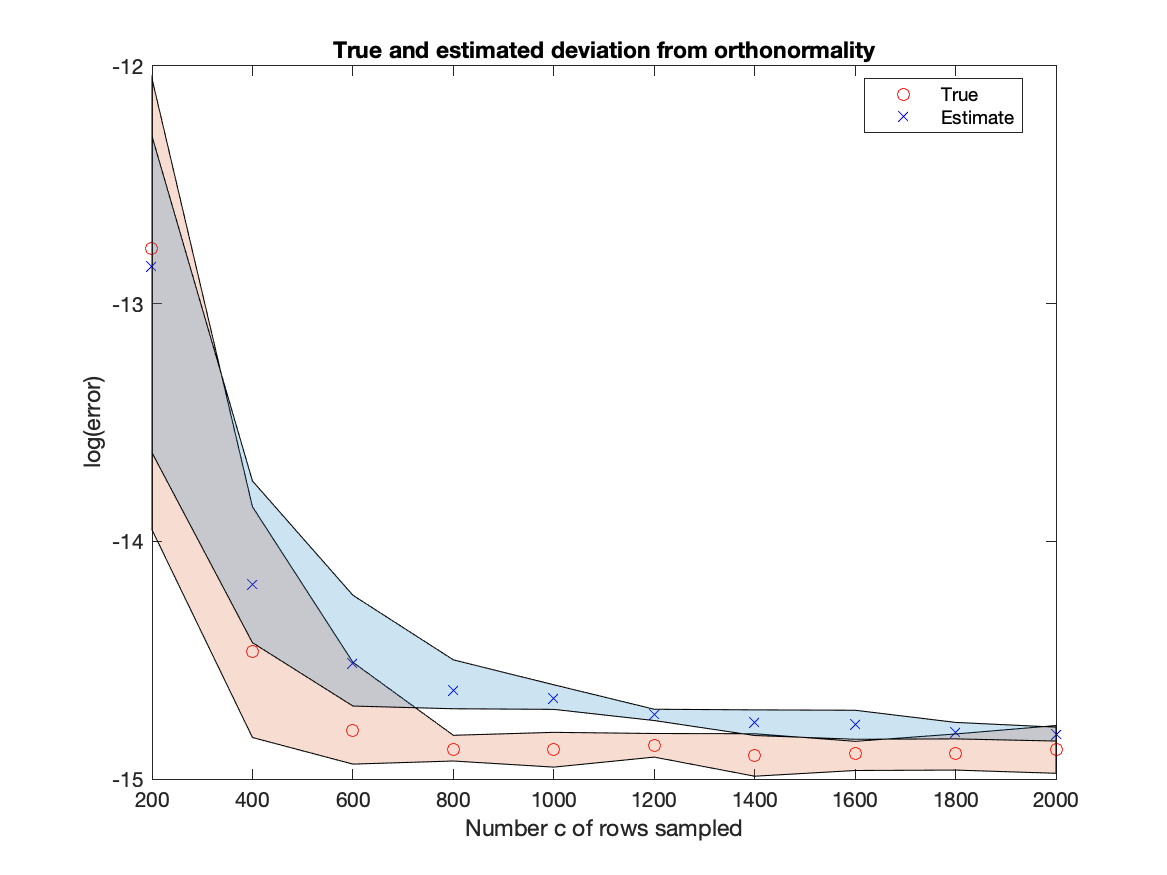}} 
\end{center}
\caption{Logarithm of deviation from orthonormality and estimate (\ref{e_ex1}) versus
sampling amount~$c$ for $\ma\in\real^{6,000\times 100}$ in (\ref{e_ErrorRange}).
The red region delineates
the smallest and largest deviation from orthonormality $\|\widehat{\mq}^T\widehat{\mq}-\mi\|_2$ 
over 10 runs for each sampling amount~$c$, and the red circles represent the mean.
The blue region delineates the smallest and largest estimate~(\ref{e_ex1})
 over 10 runs for each sampling amount~$c$, 
and the blue crosses represent the mean.}
\end{figure}

\begin{figure}\label{fig_OrthoDev1}
\begin{center}
\resizebox{2.9in}{!}
{\includegraphics{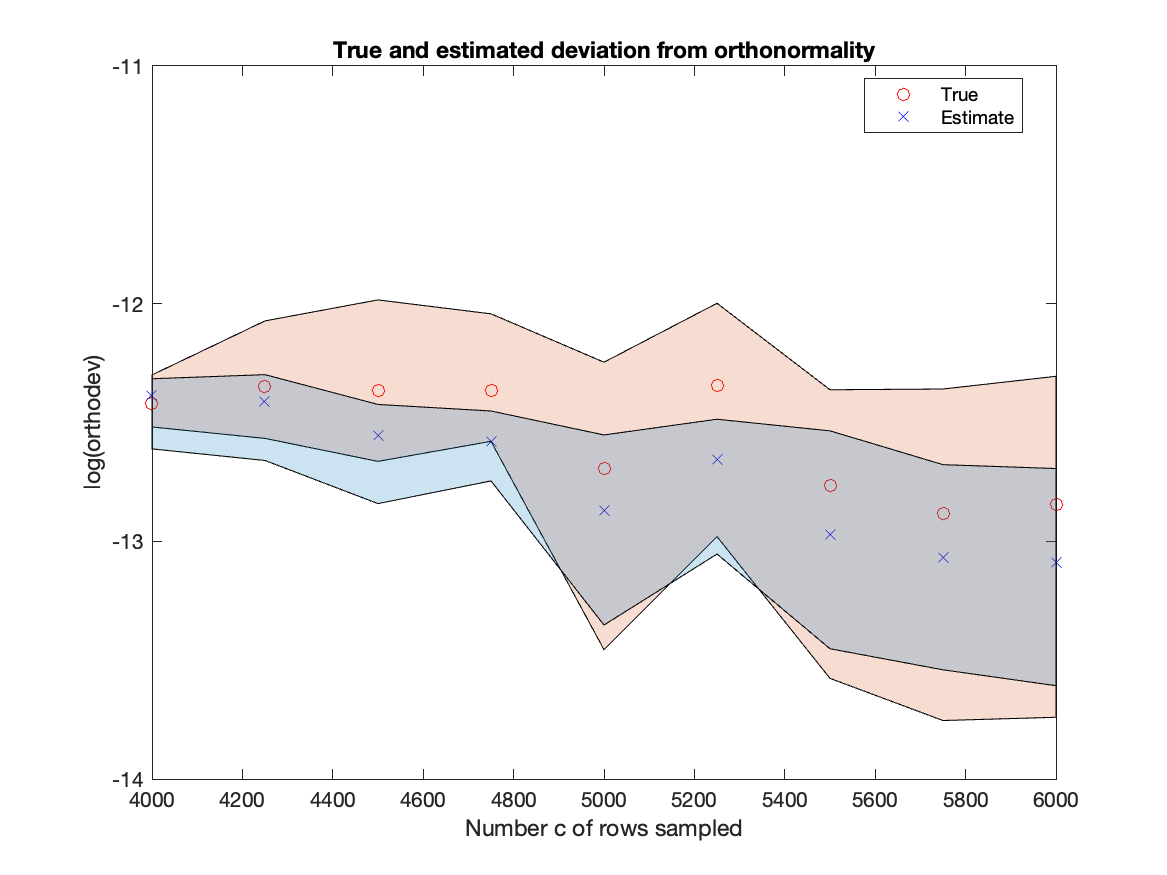}} 
\end{center}
\caption{Logarithm of deviation from orthonormality and estimate (\ref{e_ex1}) versus 
sampling amount~$c$ for
$\ma\in\real^{6,000\times 1,000}$
in (\ref{e_ErrorRange}) with $\kappa(\ma)=10^{15}$.
The red region delineates
the smallest and largest deviation from orthonormality $\|\widehat{\mq}^T\widehat{\mq}-\mi\|_2$ 
over 10 runs for each sampling amount~$c$, and the red circles represent the mean.
The blue region delineates the smallest and largest estimate~(\ref{e_ex1})
 over 10 runs for each sampling amount~$c$, 
and the blue crosses represent the mean.}
\end{figure}

\begin{figure}\label{fig_OrthoDevG}
\begin{center}
\resizebox{2.9in}{!}
{\includegraphics{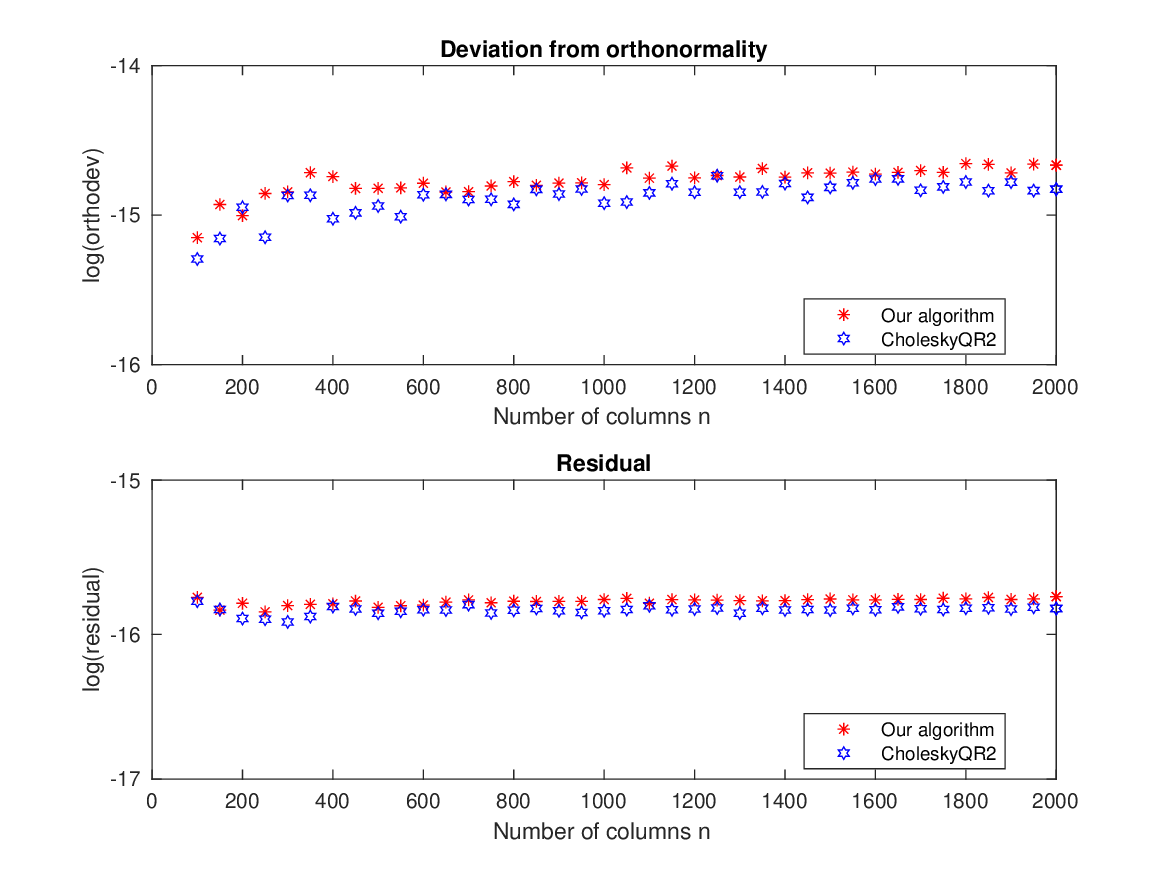}} 
\end{center}
\caption{Logarithm of deviation from orthonormality and estimate (\ref{e_ex1})
versus number of columns~$n$ for \textit{rpCholesky-QR} with $c=3n$ samples, 
applied to matrices
$\ma\in\real^{6,000\times n}$ in~(\ref{e_ErrorRange}) with $n=100, \ldots, 2,000$ columns.
The red stars represent the deviation from orthonormality
and the blue diamonds the estimate.
}
\end{figure}

\subsection{Comparison with \textit{Cholesky-QR2}}\label{s_exp3}
Figures \ref{fig_CQR2d} and~\ref{fig_CQR2G} illustrate
that \textit{rpCholesky-QR} has the same high accuracy
for moderately conditioned matrices as
\textit{Cholesky-QR2} in \cite[section 2.2]{YNYF2015}, another two-stage algorithm
with about the same operation count.

To this end, we premultiply the block upper triangular matrices (\ref{e_ErrorRange}) by a Haar
matrix, 
\begin{align}\label{e_CQR2}
\ma=\mq_A\mr_A.
\end{align}
With a condition number of $\kappa(\ma)=10^7$, the matrices $\ma$
are at the limit of \textit{Cholesky-QR2}'s capabilities. The matrices $\mq_A$ and $\mr_A$
are computed with the Matlab commands
\smallskip

\begin{verbatim}
    [H_A, ~] = qr(randn(m));
    Q_A = H_A(:, 1:n);
    R_A = gallery('randsvd', n, 10^7);
\end{verbatim}
\smallskip

\begin{figure}\label{fig_CQR2d}
\begin{center}
\resizebox{2.9in}{!}
{\includegraphics{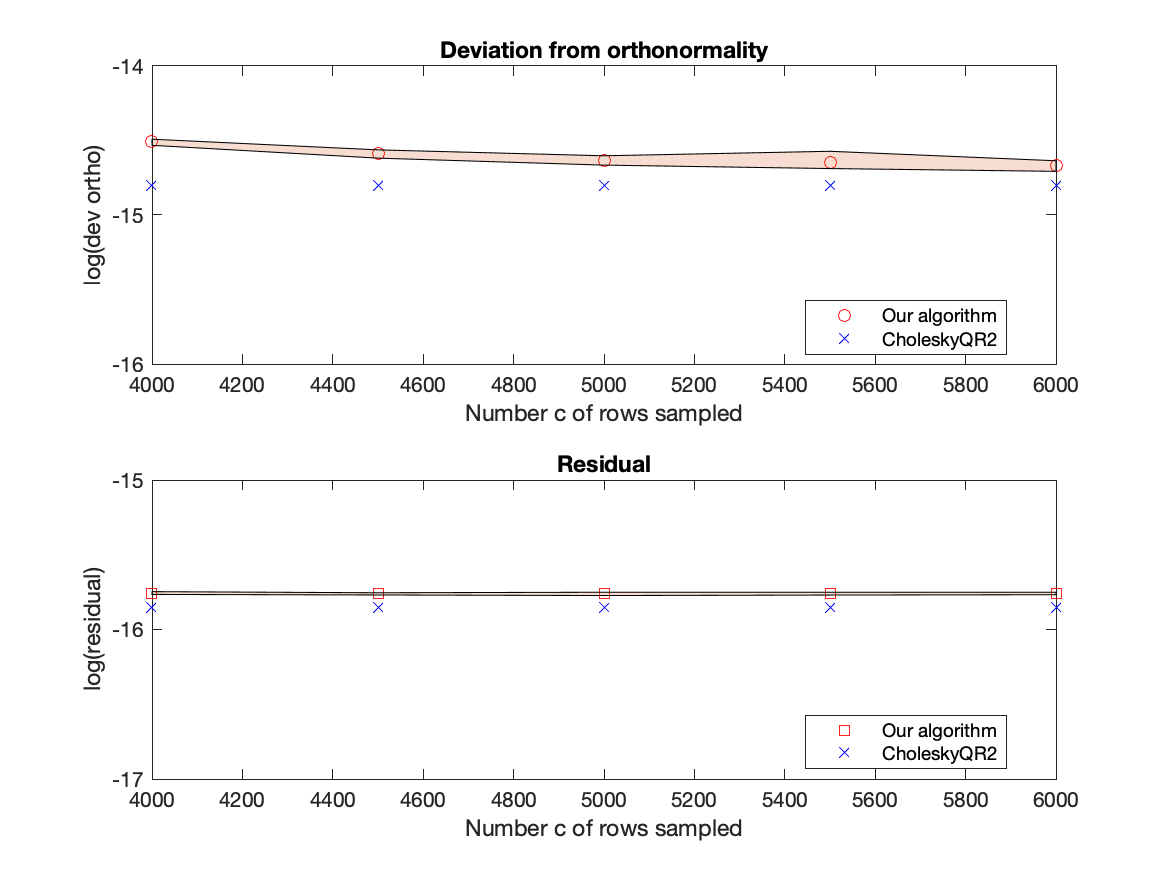}} 
\end{center}
\caption{Logarithm of deviation from orthonormality and residual versus sampling amount~$c$
for our algorithm \textit{rpCholesky-QR}
and \textit{Cholesky-QR2}, applied to 
$\ma\in\real^{6,000\times 2,000}$ in (\ref{e_CQR2}).
The upper panel shows the deviation from orthonormality 
$\|\widehat{\mq}^T\widehat{\mq}-\mi\|_2$. The red region delineates
the smallest and largest deviations
over 10 runs for each sampling amount~$c$ from \textit{rpCholesky-QR}, 
and the red circles represent the mean.
The blue crosses represent the deviations from \textit{Cholesky-QR2}.
The lower panel shows the residual
$\|\ma-\widehat{\mq}\widehat{\mr}\|_2/\|\ma\|_2$. The red region delineates
the smallest and largest residuals
over 10 runs for each sampling amount~$c$ from \textit{rpCholesky-QR}, 
and the red circles represent the mean.
The blue crosses represent the residuals from \textit{Cholesky-QR2}.
}
\end{figure}

\paragraph{Figure~\ref{fig_CQR2d}}
We compare the deviation from orthonormality and residual versus sampling amount $c$
for \textit{rpCholesky-QR} and \textit{Cholesky-QR2}. The algorithms are applied to matrices
$\ma\in\real^{6,000\times 2,000}$ in (\ref{e_CQR2}).

Even for small sampling amounts $c=3n$,
\textit{pr-Cholesky-QR} has the same accurate
deviation from orthonormality and residuals as does \textit{Cholesky-QR2}.

\begin{figure}\label{fig_CQR2G}
\begin{center}
\resizebox{2.9in}{!}
{\includegraphics{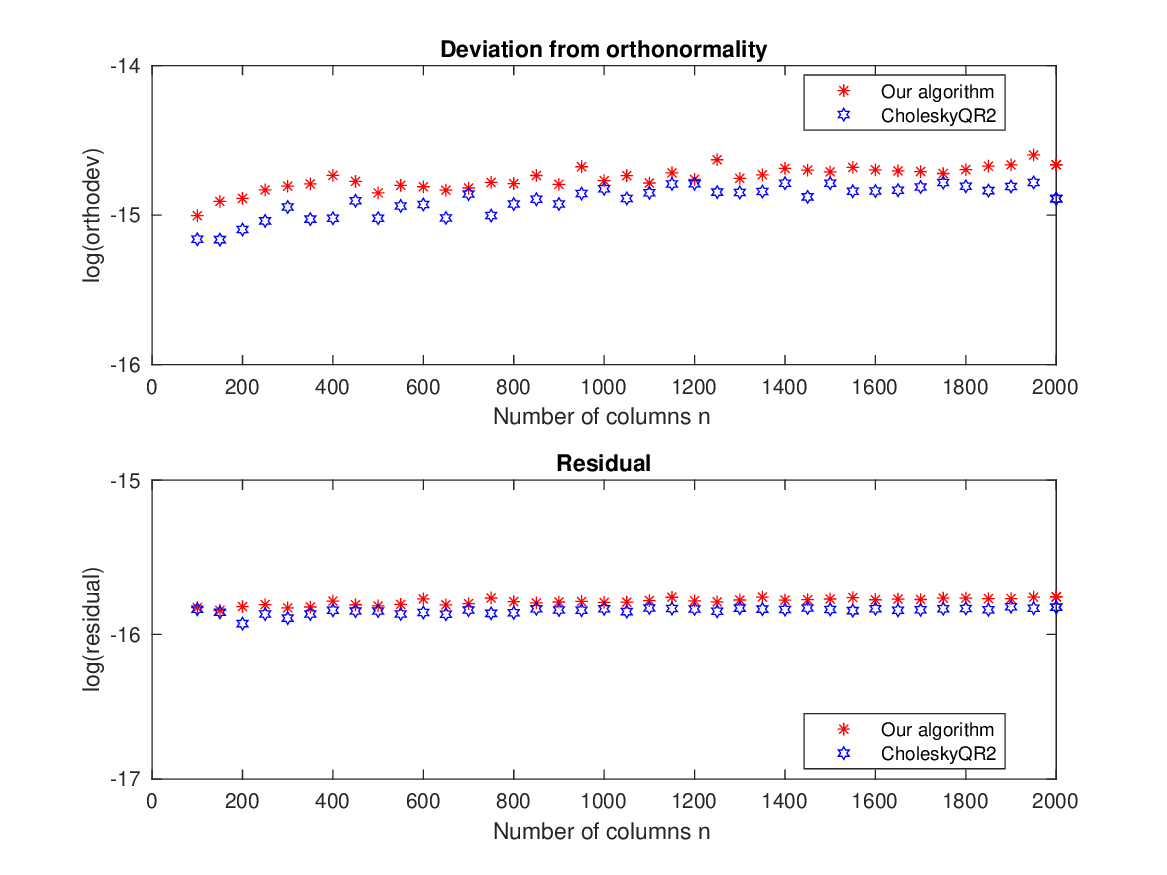}} 
\end{center}
\caption{Logarithm of deviation from orthonormality and residual versus number
of columns~$n$ for \textit{rpCholesky-QR} with $c=3n$ samples, and \textit{Cholesky-QR2}, 
applied to matrices
$\ma\in\real^{6,000\times n}$ in~(\ref{e_CQR2}) with $n=100, \ldots, 2,000$ columns.
The upper panel shows the deviation from orthonormality 
$\|\widehat{\mq}^T\widehat{\mq}-\mi\|_2$, and the lower panel the residual
$\|\ma-\widehat{\mq}\widehat{\mr}\|_2/\|\ma\|_2$.
The red stars represent \textit{rpCholesky-QR}
and the blue diamonds  \textit{Cholesky-QR2}.
}
\end{figure}

\paragraph{Figure~\ref{fig_CQR2G}}
We compare the deviation from orthonormality and residual versus number
of columns~$n$ for \textit{rpCholesky-QR} and \textit{Cholesky-QR2}.
The algorithms are applied to matrices
$\ma\in\real^{6,000\times n}$ in (\ref{e_CQR2}) with a number of columns in the range 
$n=100, \ldots, 2,000$. 
Our algorithm \textit{rpCholesky-QR} uses $c=3n$ samples.

The residuals of \textit{rpCholesky-QR} have the same magnitude,
slightly above $10^{-16}$ for all $n$,
as those of \textit{Cholesky-QR2}. The deviation from orthonormality of 
\textit{rpCholesky-QR} has the same magnitude, slightly above $10^{-15}$, as that 
of \textit{Cholesky-QR2}. 

Consequently, with only a small number $3n$ of samples, \textit{rpCholesky-QR} has the
same high accuracy for moderately conditioned matrices as does \textit{Cholesky-QR2}.

\subsection*{Acknowledgements}
We thank Laura Grigori and Arnel Smith  for helpful discussions.

\bibliography{bib}
\bibliographystyle{siam}

\end{document}